\documentclass[12pt]{amsart}
\usepackage{verbatim,amssymb,latexsym,amscd}
\usepackage[all]{xy}
\usepackage[colorlinks,linkcolor=blue,citecolor=blue,urlcolor=red]{hyperref}
\usepackage[T1]{fontenc}

\newcommand{\un}{\mathbf{1}}
\newcommand{\A}{\mathbf{A}}

\newcommand{\C}{\mathbf{C}}

\renewcommand{\H}{\mathbf{H}}
\newcommand{\N}{\mathbf{N}}
\renewcommand{\P}{\mathbf{P}}
\newcommand{\Q}{\mathbf{Q}}

\newcommand{\Z}{\mathbf{Z}}
\newcommand{\sA}{\mathcal{A}}

\newcommand{\sH}{\mathcal{H}}

\newcommand{\sK}{\mathcal{K}}

\newcommand{\sX}{\mathcal{X}}

\newcommand{\X}{{\overline X}}

\DeclareFontFamily{U}{wncy}{}
\DeclareFontShape{U}{wncy}{m}{n}{%
<5>wncyr5%
<6>wncyr6%
<7>wncyr7%
<8>wncyr8%
<9>wncyr9%
<10>wncyr10%
<11>wncyr10%
<12>wncyr6%
<14>wncyr7%
<17>wncyr8%
<20>wncyr10%
<25>wncyr10}{}
\DeclareMathAlphabet{\cyr}{U}{wncy}{m}{n}

\newcommand{\Spec}{{\operatorname{Spec \ }}}
\newcommand{\Br}{\operatorname{Br}}
\newcommand{\Ker}{\operatorname{Ker}}
\newcommand{\Coker}{\operatorname{Coker}}
\newcommand{\IM}{\operatorname{Im}}

\newcommand{\lcm}{\operatorname{lcm}}

\newcommand{\norm}{{\operatorname{norm}}}

\newcommand{\tr}{{\operatorname{tr}}}

\newcommand{\Tor}{\operatorname{Tor}}
\newcommand{\Hom}{\operatorname{Hom}}

\newcommand{\Albv}{\operatorname{Alb}}
\newcommand{\NS}{\operatorname{NS}}

\newcommand{\Pic}{\operatorname{Pic}}

\newcommand{\cl}{{\operatorname{cl}}}

\newcommand{\Ab}{\operatorname{\bf Ab}}

\newcommand{\car}{\operatorname{char}}
\newcommand{\tors}{{\operatorname{tors}}}

\newcommand{\alg}{{\operatorname{alg}}}

\newcommand{\nr}{{\operatorname{nr}}}

\newcommand{\Sm}{\operatorname{\bf Sm}}

\newcommand{\proj}{{\operatorname{proj}}}

\newcommand{\Chow}{\operatorname{\bf Chow}}

\newcommand{\et}{{\operatorname{\acute{e}t}}}
\renewcommand{\o}{{\operatorname{o}}}

\newcommand{\eff}{{\operatorname{eff}}}

\newcommand{\by}{\xrightarrow}

\newcommand{\iso}{\by{\sim}}

\newcommand{\inj}{\hookrightarrow}

\newcommand{\surj}{\rightarrow\!\!\!\!\!\rightarrow}
\newcommand{\Surj}{\relbar\joinrel\surj} 
\newcommand{\colim}{\varinjlim}
\renewcommand{\lim}{\varprojlim}

\newcommand{\gr}{\operatorname{gr}}

\renewcommand{\qed}{\hfill $\Box$\medskip}

\renewcommand{\phi}{\varphi}
\renewcommand{\epsilon}{\varepsilon}

\newcounter{spec}
\newenvironment{thlist}{\begin{list}{\rm{(\roman{spec})}}%
{\usecounter{spec}\labelwidth=20pt\itemindent=0pt\labelsep=10pt}}%
{\end{list}}%

\numberwithin{equation}{section}

\newtheorem{thm}{Theorem}[section]
\newtheorem{lemma}[thm]{Lemma}
\newtheorem{prop}[thm]{Proposition}
\newtheorem{cor}[thm]{Corollary}

\theoremstyle{definition}

\newtheorem{defn}[thm]{Definition}

\newtheorem{ass}[thm]{Assumption}
\newtheorem{rk}[thm]{Remark}
\newtheorem{rks}[thm]{Remarks}
\newtheorem{qn}[thm]{Question}

\newtheorem{exs}[thm]{Examples}

\begin{document}
\title[Torsion orders]{Torsion order of smooth projective surfaces}
\author[Bruno Kahn]{Bruno Kahn\\with an appendix by J.-L. Colliot-Th\'el\`ene}
\address{IMJ-PRG\\ Case 247\\4 place Jussieu\\
75252 Paris Cedex 05\\France}
\email{bruno.kahn@imj-prg.fr}
\address{CNRS,  Universit\'e Paris Sud\\Math\'ematiques, B\^atiment 425\\91405 Orsay Cedex\\France}
\email{jlct@math.u-psud.fr}
\date{February 27, 2017}
\begin{abstract}
To a smooth projective variety $X$ whose Chow group of $0$-cycles is $\Q$-universally trivial one can associate its torsion order $\Tor(X)$, the smallest multiple of the diagonal appearing in a cycle-theoretic decomposition à la Bloch-Srinivas.  We show that $\Tor(X)$ is the exponent of the torsion in the Néron-Severi group of $X$ when $X$ is a surface over an algebraically closed field $k$, up to a power of the exponential characteristic of $k$.
\end{abstract}
\maketitle

\tableofcontents

\subsection*{Introduction}

\enlargethispage*{20pt}

Let $X$ be a smooth projective irreducible variety over a field $k$. Assume that $CH_0(X_{k(X)})\otimes \Q\iso \Q$: this is the strongest case of ``decomposition of the diagonal'' \`a la Bloch-Srinivas \cite{bs}. To $X$ is associated its \emph{torsion order} $\Tor(X)$, the smallest multiple of the diagonal of $X$ appearing in such a decomposition (Definition \ref{dtors}). This number is also studied by Chatzistamatiou and Levine in \cite{ch-le}.

The integer $\Tor(X)$ kills all normalised motivic birational invariants of smooth projective varieties in the sense of Definition \ref{d1} (Lemma \ref{l4}). In particular, away from $\car k$, the exponent of the torsion subgroup of the geometric N\'eron-Severi group of $X$ divides $\Tor(X)$ (Corollary \ref{c1}); the main result of this paper is that we have equality when $X$ is a surface and $k$ is algebraically closed: this result was announced in \cite[Remark 3.1.5 3)]{birat-pure}. In the special case where $\Tor(X)=1$, it was obtained previously in \cite{voisin2} and \cite{auel} (see also Theorem \ref{torsion} in the appendix).

The equality follows from a short exact sequence (Corollary \ref{c1} a)):
\begin{multline}\label{eq00}
0\to  CH^2(X_{k(X)})_\tors \to \Tor(H^2(X),H^3(X))^2\\
\to H^3_\nr(X\times X,\Q/\Z(2))\to 0
\end{multline}
where $H^*(X)$ is Betti cohomology of $X$ with integer coefficients in characteristic $0$ (for simplicity; in positive characteristic, use $l$-adic cohomology). It also shows that $CH^2(X_{k(X)})_\tors$ is finite (away from the characteristic of $k$), with a very explicit bound.\footnote{It would be interesting to completely determine $CH^2(X_{k(X)})_\tors$: for example, when $X$ is an Enriques surface and $\car k=0$, is it $\Z/2$ or $(\Z/2)^2$?} 

The exact sequence \eqref{eq00} is a special case of a more general one appearing in Theorem \ref{p2}, which implies in particular the finiteness of $CH^2(X_{k(Y)})_\tors$ for any other smooth projective $Y$, and an explicit bound on its order. See Theorem \ref{finitude} for another proof of this finiteness, and a different bound.

\subsection*{Acknowledgements} I thank Arnaud Beauville, Luc Illusie and especially Jean-Louis Colliot-Thélène for helpful exchanges. (See Remark \ref{r1} for further comments.)

\section{Basic properties of the torsion order} 

\subsection{Review of birational motives} We fix a base field $k$, and write $\Sm^\proj=\Sm^\proj(k)$ for the category of smooth projective $k$-varieties. Recall from \cite{birat-pure} the category $\Chow^\o(k,A)$ of birational Chow motives with coefficients in a commutative ring $A$: there is a commutative diagram of functors
\[\xymatrix{
\Sm(k)\ar[r]^-{h}\ar[rd]^-{h^\o}&\Chow^\eff(k,A)\ar[d]\\
&\Chow^\o(k,A)
}\]
where $\Chow^\eff(k,A)$ is the covariant category of effective Chow motives with coefficients in $A$ (opposite to that of \cite{scholl}), and Hom groups in $\Chow^\o(k,A)$ are characterized by the formula
\[\Chow^\o(k,A)(h^\o(Y),h^\o(X))=CH_0(X_{k(Y)})\otimes A\]
for $X,Y\in \Sm^\proj(k)$ (with $Y$ irreducible). When $A=\Z$, we simplify the notation to $\Chow^\eff(k),\Chow^\o(k)$, or even $\Chow^\eff,\Chow^\o$.

\subsection{Motivic birational invariants} Let $X\in \Sm^\proj(k)$ be irreducible, with $CH_0(X_{k(X)})\allowbreak\otimes \Q\iso \Q$: this condition is equivalent to Bloch-Srinivas' decomposition of the diagonal relative to a closed subset of dimension $0$ \cite{bs}. By \cite[Prop. 3.1.1]{birat-pure}, this means that the birational motive $h^\o(X)$ of $X$ in the category $\Chow^\o(k,\Q)$ is trivial, i.e., that the projection map $h^\o(X)\to h^\o(\Spec k)=:\un$ is an isomorphism in $\Chow^\o(k,\Q)$. Then $CH_0(X_K)\allowbreak\otimes \Q\iso \Q$ for any field extension $K$ of $k$ (loc. cit., Condition (vi)). 

To such an $X$, we want to associate a numerical invariant. To motivate it, let us introduce a definition:

\begin{defn}\label{d1} A \emph{motivic invariant} of smooth projective varieties with values in an additive category $\sA$ is a functor $F:\Sm^\proj\to \sA$ which factors through an additive functor $\Chow^\eff\to \sA$. We say that $F$ is \emph{birational} if it further factors through $\Chow^\o$. The invariant $F$ is \emph{normalised} if $F(\Spec k)=0$.
\end{defn}

\begin{rk} If $X,Y\in \Sm^\proj$ are (stably) birationally equivalent, then $h^\o(X)\simeq h^\o(Y)$ in $\Chow^\o$ \cite[Prop. 2.3.8]{birat-pure}. Hence, to be a motivic birational invariant is stronger than to be a (stable) birational invariant. It is much stronger: $h^\o(S)\iso \un$ for $S$ the Barlow surface \cite{barlow1}, a complex surface of general type.
\end{rk}

\begin{exs}\label{ex0} a) For any cycle module $M_*$ in the sense of Rost \cite{rost}, any $K\supseteq k$ and any $n\in\Z$, $X\mapsto A^0(X_K,M_n)$ (resp. $X\mapsto A_0(X_K,M_n)$) defines a contravariant (resp. covariant) motivic birational invariant with values in $\Ab$, the category of abelian groups \cite[Cor. 6.1.3]{birat-pure}.\\ 
b) In particular, for $M_*=K_*^M$ (Milnor $K$-theory), the functor $X\mapsto A_0(X_K,M_0)=CH_0(X_K)$ is a motivic birational invariant. When $K=k(Y)$ for some $Y\in \Sm^\proj$, this is also obvious by the intepretation of $CH_0(X_K)$ as $\Chow^\o(h^\o(Y),h^\o(X))$.\\
c) Given a contravariant motivic invariant $F$, we get two (contravariant) normalised invariants by the formulas
\[\underline{F}(X) = \Ker(F(k)\to F(X)), \quad \overline{F}(X) =  \Coker(F(k)\to F(X))\]
and similarly for covariant motivic invariants:
\[\underline{F}(X) = \Coker(F(X)\to F(k)), \quad \overline{F}(X) =  \Ker(F(X)\to F(k)).\]
They are birational if $F$ is birational.\\
d) Suppose that $F$ is a motivic invariant with values in the category of $\Z[1/p]$-modules, where $p$ is the exponential characteristic of $k$ (or, more generally, in a $\Z[1/p]$-linear additive category); assume $F$ contravariant to fix ideas. Then $F$ is birational if and only if, for any $Y\in \Sm^\proj$, the map $F(Y)\to F(Y\times \P^1)$ is an isomorphism. This follows from \cite[Th. 2.4.2]{birat-pure}.
\end{exs}

\begin{defn} The category $\Chow^\o_\norm$ is the quotient of $\Chow^\o$ by the ideal generated by $\un$.
\end{defn}

Thus a motivic birational invariant is normalised if and only if it factors through $\Chow^\o_\norm$.

Let $M,N\in \Chow^\o$. By definition, $\Chow^\o_\norm(M,N)$ is the quotient of $\Chow^\o(M,N)$ by the group of morphisms $f:M\to N$ which factor through $\un$. If $M=h^\o(Y)$ and $N=h^\o(X)$, this gives
\[\Chow^\o_\norm(h^\o(Y),h^\o(X))\simeq \Coker(CH_0(X)\to CH_0(X_{k(Y)})).\]

\subsection{The torsion order} If now the birational motive of $X$ is trivial in $\Chow^\o(k,\Q)$, then the image of $h^\o(X)$ in $\Chow^\o_\norm$ is torsion; in other words, there is an integer $n>0$ such that $n 1_{h^\o(X)}=0$ in $\Chow^\o_\norm(h^\o(X),h^\o(X))$.

\begin{defn}\label{dtors} The smallest such integer $n$ is called the \emph{torsion order} of $X$, and denoted by $\Tor(X)$. We extend this to arbitrary (connected) $X$ by setting $\Tor(X)=0$ if $h^\o(X)$ is not trivial in $\Chow^\o(k,\Q)$.\\
If $p$ is the exponential characteristic of $k$, we write $\Tor^p(X)$ for the part of $\Tor(X)$ which is prime to $p$ (so $\Tor^p(X)=\Tor(X)$ if $\car k=0$).
\end{defn}

In $\Chow^\o$, the identity morphism $1_{h^\o(X)}$ is given by $\eta_X\in CH_0(X_{k(X)})$, where $\eta_X$ is the generic point viewed as a closed point of $X_{k(X)}$. This gives a concrete description of the torsion order: 

\begin{lemma}\label{l4} Suppose that $CH_0(X_{k(X)})\otimes \Q\iso \Q$. Then  the torsion order of $X$ is the order $n$ of $\eta_X$ in the group $CH_0(X_{k(X)})/CH_0(X)$ (it is $0$ if and only if $\eta_X$ has infinite order). Moreover, we have $nF(X)=0$ for any normalised motivic birational invariant $F$. In particular,
\[nCH_0(X_K)_0=n \frac{CH_0(X_K)}{CH_0(X)}
=0 \quad \text{for any } K\supseteq k\]
where $CH_0(X_K)_0=\Ker(CH_0(X_K)\by{\deg} \Z)$.
\end{lemma}

\begin{proof} The first and second statements are tautological; the third follows as a special case of the second.
\end{proof}

\subsection{Torsion order and index} Another important invariant is

\begin{defn} The \emph{index} of an irreducible $X\in \Sm^\proj$ is the positive generator of $\IM(\deg:CH_0(X)\to \Z)$. We denote it by $I(X)$.
\end{defn}

\begin{prop} Let $X\in \Sm^\proj$, irreducible. Write $n$ for its torsion order and $d$ for its index.\\
a) If $F$ is a motivic invariant and $\underline{F}$ is as in Example \ref{ex0} c), then we have $d\underline{F}(X)=0$.\\
b) $n$ is divisible by $d$.\\
c) Suppose $CH_0(X_{k(X)})\otimes \Q\iso \Q$. If $x\in CH_0(X)$ is an element of degree $d$, then $m(x_{k(X)}-d\eta_X)=0$ in $CH_0(X_{k(X)})$ for some $m>0$, and $n\mid md$.\\
d) If $d=1$ and $m$ is minimal in c), then $n=m$.
\end{prop}

\begin{proof} a) Suppose $F$ is contravariant. Let $\alpha\in F(k)$ be such that $\pi^*\alpha=0$, where $\pi:X\to \Spec k$ is the structural morphism. If $x\in CH_0(X)$ is an element of degree $d$, it defines a morphism $x:\un\to h^\o(X)$ such that $\pi\circ x = d$. Hence $d\alpha=0$. 

b) A diagram chase yields an exact sequence
\begin{equation}\label{eq6}
CH_0(X)_0\to CH_0(X_{k(X)})_0\to \frac{CH_0(X_{k(X)})}{CH_0(X)}\to \Z/d\to 0
\end{equation}
where $CH_0(X_K)_0$ was defined in Lemma \ref{l4} and the last map sends the class of $\eta_X$ to $1$.

c) The first claim follows from \eqref{eq6}, and the second follows from pushing this identity in $CH_0(X_{k(X)})/CH_0(X)$.

d) If $d=1$, \eqref{eq6} yields a surjection
\[CH_0(X_{k(X)})_0\Surj \frac{CH_0(X_{k(X)})}{CH_0(X)}.\]

Let $y\in CH_0(X_{k(X)})_0$ mapping to the class of $\eta_X$. This means that $\eta_X-y =x_{k(X)}$ for some $x\in CH_0(X)$, and necessarily $\deg(x)=1$. By Lemma \ref{l4}, we have $ny=0$ so the conclusion is true for this choice of $x$. But if $x'\in CH_0(X)$ is of degree $1$, then $n(x'-x)=0$ hence the conclusion remains true when replacing $x$ by $x'$.
\end{proof}

\begin{rk} When $d=1$, we can avoid the recourse to the category $\Chow^\o_\norm$: in this case, the morphism $h^\o(X)\to \un$ is (noncanonically) split, hence we may consider its kernel $h^\o(X)_0\in \Chow^\o$. The endomorphism ring of this birational motive is canonically isomorphic to $CH_0(X_{k(X)})/CH_0(X)$.
\end{rk}

\subsection{Change of base field and products}

\begin{prop} Let $K/k$ be a field extension. Then:\\
a) $\Tor(X_K)\mid \Tor(X)$.\\
b) If $k$ and $K$ are algebraically closed, then $\Tor(X_K)= \Tor(X)$.
\end{prop}

\begin{proof} a) is obvious, and b) follows from the rigidity theorem for torsion in Chow groups \cite{lecomte}.
\end{proof}

\begin{prop} For any connected $X,Y\in \Sm^\proj$, $\Tor(X\times Y)\mid \Tor(X)\Tor(Y)$.
\end{prop}

\begin{proof} If $\Tor(X)=0$ or $\Tor(Y)=0$, this is obvious. Otherwise, let $m>0$ (resp. $n>0$) be such that $m1_{h^\o(X)}$ (resp. $n1_{h^\o(Y)}$) factors through $\un$. Then $mn1_{h^\o(X\times Y)}=m1_{h^\o(X)}\otimes n1_{h^\o(Y)}$ factors through $\un\otimes \un=\un$. 
\end{proof}

\section{Torsion order for cycle modules} 

For any abelian group $A$, write
\[\exp(A)=\inf\{m>0\mid mA=0\}\]
and, by convention, $\exp(A)=0$ if no such integer $m$ exists. Also write
\[\exp^p(A)=\exp(A[1/p]).\]

\subsection{General case}  We refine the notion of torsion order as follows:

\begin{defn}\label{d2.1} Let $M$ be a cycle module. For $X\in \Sm^\proj$, $K\supseteq k$ and $n\in\Z$, write $F_n(X_K) = \Coker(M_n(K)\to A^0(X_K,M_n))$: then $X\mapsto F_n(X_K)$ is a normalised motivic birational invariant in the sense of Definition \ref{d1}. We set
\begin{align*}
\Tor_K(X,M_n) &= \exp(F_n(X_K)),\\
\Tor(X,M_n)&=\lcm_{K\supseteq k} \Tor_K(X,M_n),\\
\Tor(X,M)&=\lcm_{n} \Tor(X,M_n).
\end{align*}
where $\lcm$ means lower common multiple.\end{defn}

By Lemma \ref{l4}, $\Tor_K(X,M_n)\mid\Tor(X,M_n)\mid \Tor(X)$. Moreover,

\begin{lemma} $ \Tor(X,M_{n-1})\mid \Tor(X,M_n)$.
\end{lemma}

\begin{proof} Let $K/k$ be an extension. We have a naturally split exact sequence (\cite[Prop. 2.2]{rost} and its proof):
\[0\to A^0(X_K,M_n)\to A^0(X_{K(t)},M_n)\to \bigoplus_{x\in (\A^1_K)^{(1)}} A^0(X_{K(x)},M_{n-1})\to 0.\]

Indeed, $K\mapsto A^0(X_K,M_n)$ defines a cycle module. Comparing with the same exact sequence for $X=\Spec k$, we get the conclusion.
\end{proof}

\subsection{Unramified cohomology of degree $\le 2$}\label{s2.2} For $K\supseteq k$, we write $\bar K$ for an algebraic closure of $K$ and $G_K=Gal(\bar K/K)$. Let $p$ be the exponential characteristic of $k$. We compute $\Tor(X,\sH_n)$ for low values of $n$, where $\sH_n$ is the cycle module $K\mapsto H^n_\et(K,(\Q/\Z)'(n-1))$ with $(\Q/\Z)'(n-1):=\colim_{(m,p)=1} \mu_m^{\otimes n-1}$. As is well-known, we have
\[A^0(X_K,\sH_n)=
\begin{cases}
H^0(K,(\Q/\Z)'(-1)) &\text{for $n=0$}\\
H^1(X_K,(\Q/\Z)') &\text{for $n=1$}\\
\Br(X_K)[1/p]&\text{for $n=2$.}
\end{cases}
\]

 Let $X$ be such that $CH_0(X_{k(X)})\otimes \Q\iso \Q$;  then $b^1(X)=0$ and $b^2(X)=\rho(X)$ where $b^i(X)$ (resp. $\rho(X)$) denotes the $i$-th Betti number (resp. the Picard number) of $X$ \cite[Prop. 3.1.4 3)]{birat-pure}. In particular, we have  $\Pic^0_{X/k}=0$ and for any $K\supseteq k$, $H^1(X_{\bar K},(\Q/\Z)')\iso \NS(X_{\bar k})_\tors[1/p]$ and similarly $\Br(X_{\bar K})\{l\}\allowbreak\iso H^3_\et(X_{\bar K},\Z_l)_\tors$ for $l\ne p$, so $\Br(X_{\bar k})[1/p]\iso \Br(X_{\bar K})[1/p]$. (We neglected Tate twists in these computations.)
 
In the sequel, we abbreviate $X_{\bar k}$ to $\X$; for simplicity, we assume $I(X)=1$ so that $H^i(K,(\Q/\Z)'(j))\to H^i(X_K,(\Q/\Z)'(j))$ is split injective for any $K,i,j$. The Hochschild-Serre spectral sequence then gives isomorphisms (see Definition \ref{d2.1} for the notation $F_n$):
\[F_0(X_K)=0,\quad F_1(X_K) = (\NS(\X)_\tors[1/p])^{G_K} \]
and an exact sequence
\begin{equation}\label{eq8}
0\to H^1(K,\NS(\X))[1/p]\to F_2(X_K)\to (\Br(\X)[1/p])^{G_K}.
\end{equation}

For $K\supseteq \bar k$, $G_K$ acts trivially on $\NS(\X)$ and $\Br(\X)$. Then $H^1(K,\NS(\X))\allowbreak =\Hom(G_K,\NS(\X)_\tors)$ and the last map in \eqref{eq8} is split surjective: indeed, $\Br(\X)[1/p]$ maps to $F_2(X_K)$ by functoriality. This yields:

\begin{prop}\label{p1} Let $X$ be such that $I(X)=1$ and $CH_0(X_{k(X)})\otimes \Q\iso \Q$. Then 
\begin{align*}
\Tor(X,\sH_0)&=1\\
\Tor(X,\sH_1)&=\exp^p(\NS(\X)_\tors)\\
\Tor(X,\sH_2)&= \lcm(\exp^p(\NS(\X)_\tors),\exp^p(\Br(\X))
\end{align*}
 In particular, $\Tor(X)$ is divisible by $\exp^p(\NS(\X)_\tors)$ and $\exp^p(\Br(\X))$. \qed
\end{prop}

(Of course, one could recover this conclusion directly by considering the normalised motivic birational functors $X\mapsto \NS(\X)_\tors[1/p]$ and $X\mapsto \Br(\X)[1/p]$.)

\begin{rk} When $k$ is algebraically closed, the above computation yields $\Tor_k(X,\sH_1)=\exp^p(\NS(X)_\tors)$ and $\Tor_k(X,\sH_2)=\exp^p(\Br(X))$.
\end{rk}

When $\dim X=2$, $\exp^p(\NS(\X)_\tors)=\exp^p(\Br(\X))$ by Poincar\'e duality. We shall see in Corollary \ref{c1} that, then, $\Tor^p(\X)= \exp^p(\NS(\X)_\tors)\allowbreak=\exp^p(\Br(\X))$. In view of Proposition \ref{p1}, this also yields 
\begin{equation}\label{eq14}
\Tor^p(X)=\Tor(X,\sH) \quad \text{ if } \dim X\le 2.
\end{equation}

\begin{qn} Is the equality  \eqref{eq14} true in general? In other words, does the cycle module $\sH_*$ always compute the torsion index?
\end{qn} 

\section{Extension of functors} 

\begin{defn} \label{d0} If $F$ is a contravariant functor from smooth $k$-schemes of finite type to abelian groups, we extend it to smooth $k$-schemes essentially of finite type by the formula
\begin{equation}\label{eq5}
\tilde F(X)=\colim_\sX F(\sX)
\end{equation}
where $\sX$ runs through the smooth models of finite type of $X/k$. 
\end{defn}

Note that if $F(X)=A^n_\alg(X)$, then $F$ is defined on all smooth $k$-schemes (not necessarily of finite type), but does not commute with filtering colimits; so the natural map
\[\tilde A^n_\alg(X)\to A^n_\alg(X)\]
is not an isomorphism in general, see \cite[Rk. 2.3.10 2)]{birat-pure}. By contrast, we have:

\begin{lemma}\label{l0} For any cycle module $M$, the functors $A^p(-,M_q)$ of \S \ref{s4} below commute with filtering colimits of smooth schemes.
\end{lemma}

\begin{proof} This is obvious, since the same is true for the cycle complexes of \cite{rost}.
\end{proof}

As a special case, one recovers the commutation of Chow groups with filtering colimits \cite[Lemma 1A.1]{bloch}.

\section{The Rost spectral sequence}\label{s4} Let $M$ be a cycle module. For any smooth $X/k$, recall its \emph{cycle cohomology with coefficients in $M$}:
\[A^p(X,M_q)=H^p(\dots\to \bigoplus_{x\in X^{(p)}} M_{q-p}(k(x))\to \dots)\]
where the differentials are defined through Rost's axioms. We assume:
\begin{thlist}
\item $M_n=0$ for $n<0$;
\item $M_0(K)=A$ for any $K/k$, where $A$ is a torsion-free abelian group.
\end{thlist}

By Rost's axioms \cite{rost}, there is then a canonical homomorphism of cycle modules
\[K^M\otimes A\to M\]
where $K^M$ is the cycle module given by Milnor's $K$-theory. 
For any $n\ge 0$, this yields a \emph{surjective homomorphism}
\begin{equation}\label{eq4.1}
CH^n(X)\otimes A=A^n(X,K_n^M\otimes A)\Surj A^n(X,M_n)=:A^n_M(X).
\end{equation}

We may thus think of $A^n_M(X)$ as the group of cycles of codimension $n$ modulo ``$M$-equivalence''.

\begin{exs}\label{ex1} 1) For $M=K^M\otimes A$, we get $A^n_M(X)=CH^n(X)\otimes A$.\\
2) Let $H$ be Betti cohomology (in characteristic $0$) or $l$-adic cohomology (in characteristic $\ne l$): in the first case, let $A=\Z$ and in the second case let $A=\Z_l$. For a function field $K/k$, set
\[\H_n(K) := \tilde H^n(\Spec K,A(n))\]
see Definition \ref{d0}. (This is not the cycle module $\sH$ considered in Subsection \ref{s2.2}.) By \cite[Th. 7.3]{bo} and \cite[proof of Prop. 4.5]{cycletale}, one has
\[A^n_\H(X)=A^n_\alg(X)\otimes A\]
where $A^n_\alg(X)$ is the group of cycles of codimension $n$ on $X$, modulo algebraic equivalence.
\end{exs}

We now take two smooth $k$-varieties $X,Y$, and study the Rost spectral sequence \cite[Cor. 8.2]{rost} attached to the first projection $\pi:Y\times X\to Y$:
\begin{equation}\label{eqrost1}
E_1^{p,q}(r) = \bigoplus_{y\in Y^{(p)}} A^q(X_{k(y)},M_{r-p})\Rightarrow A^{p+q}(Y\times X,M_r)
\end{equation}
abutting to the coniveau filtration on $A^{p+q}(Y\times X,M_r)$ with respect to $Y$. Note that $A^q(X_{k(y)},M_{r-p})=0$ for $p+q>r$ by Condition (i) on $M$, hence $E_1^{p,q}(r)=0$ in that range.  

Take $r=2$: we only have to consider $p+q\le 2$.  By definition, we have for a function field $K/k$ (see \eqref{eq4.1} for the notation $A^q_M$):
\[A^q(X_K,M_q) =
\colim_U A^q_M(X\times U)=: \tilde A^q_M(X_K)
\]
where $U$ runs through smooth models of $K$ as above (see Lemma \ref{l0}).  This yields 
\begin{align*}
E_2^{0,2}(2)&= \tilde A^2_M(X_{k(Y)}) \\
E_2^{1,1}(2) &= \Coker\Big(A^1(X_{k(Y)},M_2)\to \bigoplus_{y\in Y^{(1)}} \tilde A^1_M(X_{k(y)})\Big)\\ 
E_2^{2,0}(2) &= \Coker\Big(\bigoplus_{y\in Y^{(1)}} A^0(X_{k(y)},M_1)\to Z^2(Y)\otimes A\Big).
\end{align*}

The latter group is a quotient of $A^2_M(Y)$ (consider the maps $M_1(k(y))\allowbreak \to A^0(X_{k(y)},M_1)$). If $X$ has a $0$-cycle of degree $1$, the map $A^2_M(Y)\to A^2_M(Y\times X)$ is split, hence $\pi^*:A^2_M(Y)\to E_2^{2,0}(2)$ is an isomorphism. Thus $E_2=E_\infty$ in the Rost spectral sequence. We summarise this discussion:

\begin{prop}\label{p0} Let $\gr^*_Y A^2_M(X\times Y)$ be the associated graded to the coniveau filtration relative to $Y$. Assume that $X$ has a $0$-cycle of degree $1$. Then we have isomorphisms
\begin{align*}
\gr^0_Y A^2_M(X\times Y)&= \tilde A^2_M(X_{k(Y)})\\
\gr^1_Y A^2_M(X\times Y)&= \Coker(A^1(X_{k(Y)},M_2)\to \bigoplus_{y\in Y^{(1)}} \tilde A^1_M(X_{k(y)}))\\
\gr^2_Y A^2_M(X\times Y)&= A^2_M(Y).
\end{align*}
Moreover, we have an exact sequence:
\begin{multline}\label{eq6a}
0\to A^1(Y,M_2)\to A^1(Y\times X,M_2)\to A^1(X_{k(Y)},M_2)\\
\to \bigoplus_{y\in Y^{(1)}} \tilde A^1_M(X_{k(y)})\to A^2_M(Y\times X)/A^2_M(Y).
\end{multline}
\end{prop}

\section{Trivial birational motives of surfaces}

We start with a special case of Proposition \ref{p0}:

\begin{thm}\label{t1} Suppose $k$ algebraically closed, and let $X/k$ be a smooth projective variety such that $\Pic^0_{X/k}\allowbreak =0$.
Then for any smooth $Y$, there is an exact sequence
\begin{multline}\label{eq0}
CH^2(Y)\oplus  \Pic(Y)\otimes\NS(X)\oplus CH^2(X)\\
\to CH^2(Y\times X)\to CH^2(X_{k(Y)})/CH^2(X) \to 0
\end{multline}
where the maps $CH^2(X),CH^2(Y)\to CH^2(Y\times X)$ are induced by the two projections, and the map $\Pic(Y)\otimes\NS(X)\to CH^2(Y\times X)$ is given by the cross-product of cycles. 
\end{thm}

A version of this theorem is found in Merkurjev's appendix \cite{merkpreprint}; I thank J.-L. Colliot-Th\'el\`ene for pointing out this reference.

\begin{proof} Consider the Rost spectral sequence \eqref{eqrost1} for the cycle module $M=K^M$. Since $\Pic^0_{X/k}=0$, we have $\NS(X)\iso \Pic(X_{k(y)})$ for any $y\in Y^{(1)}$, hence
\[E_2^{1,1} = \Coker(A^1(X_{k(Y)},K_2)\to Z^1(Y)\otimes \NS(X)). \]

Then the natural map $k(Y)^*\otimes \NS(X)\to A^1(X_{k(Y)}),K_2)$ realises $E_2^{1,1}$ as a quotient of $\Pic(Y)\otimes \NS(X)$. We conclude by applying Proposition \ref{p0}.
\end{proof}

Theorem \ref{t1} may be compared  with a computation of the cohomology of $Y\times X$. We  use $l$-adic cohomology, neglecting Tate twists: so $H^i(X):=\prod_{l\ne p} H^i_\et(X,\Z_l)$, where $p$ is the exponential characteristic of $k$ ($p=1$ if $\car k=0$). If $k=\C$, we have $H^i(X)\simeq H^i_B(X)\otimes\prod_{l} \Z_l$, by M. Artin's comparison theorem. We note that the choice of a rational point of $X$ gives a retraction of the map $F(Y)\to F(Y\times X)$ for any contravariant functor $F:\Sm^\proj\to \Ab$; the quotient $F(Y\times X,Y)$ is therefore a direct summand of $F(Y\times X)$. Then the K\"unneth formula gives split exact sequences
\begin{equation}\label{eq1}
0\to H^3(X)\to H^3(Y\times X,Y)\\
 \to \Tor(H^2(Y),H^2(X))\to 0
\end{equation}
and
\begin{multline}\label{eq2}
0\to H^2(Y)\otimes H^2(X)\oplus H^1(Y)\otimes H^3(X)\oplus H^4(X)\\
\to H^4(Y\times X,Y)\\
 \to \Tor(H^2(Y),H^3(X))\oplus \Tor(H^3(Y),H^2(X))\to 0.
\end{multline}

We now make the following

\begin{ass}\label{a1} $k$ is algebraically closed, $Y$ is projective and $X$ is a surface such that $CH_0(X_{k(X)})\otimes \Q\iso \Q$. 
\end{ass}

 Recall that, then, $\Albv(X)=\Pic^0_{X/k}\allowbreak=0$ and $CH^2(X)=\Z$ (Ro\v\i tman's theorem), so that Theorem \ref{t1} applies. Recall also that
\begin{align*}
H^1(X)&=0\\
  \NS(X)\otimes\hat \Z'&\iso H^2(X)\\
  H^3(X)&\simeq \Hom(\NS(X)_\tors, (\Q/\Z)')\\
  H^4(X)&=\hat\Z'
  \end{align*}
where $\hat\Z'=\prod_{l\ne p} \Z_l$. Thus \eqref{eq0} and \eqref{eq2} yield a commutative diagram
\begin{equation}\label{eq11}
\begin{CD}
&&\scriptstyle (\Pic(Y)\otimes\NS(X)\oplus \Z)\otimes \hat \Z'&\to& \scriptstyle CH^2(Y\times X,Y)\otimes \hat \Z'&\to& \scriptstyle \frac{CH^2(X_{k(Y)})}{CH^2(X)}\otimes \hat \Z' &\to& \scriptstyle 0\\
&&@V\psi VV @V{\cl^2_{Y\times X,Y}}VV @V\phi VV\\
\scriptstyle 0&\to&\scriptstyle H^2(Y)\otimes H^2(X)\oplus H^1(Y)\otimes H^3(X)\oplus \hat\Z'&\by{\theta}& \scriptstyle H^4(Y\times X,Y)&\to& \scriptstyle\Tor(H^2(Y),H^3(X))\oplus \Tor(H^3(Y),H^2(X))&\to& \scriptstyle 0.
\end{CD}
\end{equation}

An obvious generalisation of the exact sequence \eqref{eq6} boils down to an isomorphism
\[CH_0(X_{k(Y)})_0\iso CH_0(X_{k(Y)})/CH_0(X).\]

In \eqref{eq11}, the left vertical map $\psi$ is diagonal; its cokernel is 
\[\Coker \psi =H^2_\tr(Y)\otimes \NS(X)\oplus H^1(Y)\otimes H^3(X)\]
where $H^2_\tr(Y):=\Coker \cl^1_Y$, and its kernel is $\Pic^0(Y)\otimes \NS(X)\otimes \hat \Z'$
(we use here that $H^2_\tr(Y)$ is torsion-free). The snake lemma thus yields an exact sequence
\begin{multline}\label{eq9}
\Pic^0(Y)\otimes \NS(X)\otimes \hat \Z'\by{\alpha} \Ker \cl^2_{Y\times X,Y} \by{\beta} \Ker\phi\\
\to H^2_\tr(Y)\otimes \NS(X)\oplus H^1(Y)\otimes H^3(X)\by{\gamma} \Coker \cl^2_{Y\times X,Y} \to \Coker \phi\to 0.
\end{multline}

To go further, we use étale motivic cohomology as in \cite{cycletale};  the cycle class map $\cl^2_{X\times X}$ extends to an \'etale cycle class map \cite[(3-1)]{cycletale}:
\[
\tilde \cl^2_{Y\times X,Y}: H^4_\et(Y\times X,Y,\Z(2))\otimes \hat \Z'\to H^4(Y\times X,Y).
\]

\begin{thm}\label{p2} Under Assumption \ref{a1}, $\Ker \cl^2_{Y\times X,Y}$ and $\Ker\tilde \cl^2_{Y\times X,Y}$ are torsion-free; the exact sequence \eqref{eq9} yields a surjection
\[\Pic^0(Y)\otimes \NS(X)\otimes \hat \Z'\Surj \Ker \cl^2_{Y\times X,Y} \]
and an exact sequence of finite groups
\begin{multline}\label{eq7}
0\to \Ker\phi
\to H^2_\tr(Y)\otimes \NS(X)_\tors\oplus H^1(Y)\otimes H^3(X)\\
\to H^3_\nr(Y\times X,Y;(\Q/\Z)'(2))\to \Coker \phi\to 0
\end{multline}
where $H^3_\nr(Y\times X,Y;(\Q/\Z)'(2)):=\colim_{(m,p)=1}H^3_\nr(Y\times X,Y;\mu_m^{\otimes 2})$. In particular, $CH_0(X_{k(Y)})/CH_0(X)[1/p]\simeq CH_0(X_{k(Y)})_\tors[1/p]$ is finite.
\end{thm}

\begin{proof} This proof is ugly, mainly because the Leray spectral sequence for \'etale motivic cohomology relative to the projection $(Y\times X,Y)\to Y$ does not behave as well as the spectral sequence \eqref{eqrost1}. So, instead of comparing directly étale motivic and $l$-adic cohomology, we have to wiggle through.   

We have a commutative diagram
\begin{equation}\label{eq10}
\begin{CD}
H^3_\et(Y\times X,Y,(\Q/\Z)'(2)) @>\sim>> \colim_{(m,p)=1} H^3_\et(Y\times X,Y,\mu_m^{\otimes 2})\\
@VVV @VVV\\
H^4_\et(Y\times X,Y,\Z(2))\otimes \hat \Z'@>\tilde \cl^2_{Y\times X,Y}>> H^4(Y\times X,Y)
\end{CD}
\end{equation}
in which the right vertical map is injective, because $H^3(Y\times X,Y)$ is torsion by \eqref{eq1}. Thus $\Ker\tilde \cl^2_{Y\times X,Y}$ is torsion-free, and so is its subgroup $\Ker \cl^2_{Y\times X,Y}$. But the image of $\alpha$ in \eqref{eq9} is divisible, hence a direct summand. Therefore the image of $\beta$ is torsion-free, hence $0$. So we get the surjection promised in the theorem, and an exact sequence
\begin{multline}\label{eq13}
0\to \Ker\phi
\to H^2_\tr(Y)\otimes \NS(X)\oplus H^1(Y)\otimes H^3(X)\\
\to \Coker \cl^2_{Y\times X,Y} \to \Coker \phi\to 0.
\end{multline}

As a consequence, $\Ker\phi$ is finitely generated; since it is torsion it must be finite, hence  $CH_0(X_{k(Y)})/CH_0(X)[1/p]$ is finite. 

We now deduce from  \cite[Th. 1.1]{cycletale} the following surjection:
\begin{equation}\label{eq12}
H^3_\nr(Y\times X,Y;(\Q/\Z)'(2))\Surj (\Coker \cl^2_{Y\times X,Y})_\tors
\end{equation} 
(if $k=\C$, this is due to Colliot-Thélène--Voisin \cite[Th. 3.7]{ctv}, with Betti cohomology instead of $l$-adic cohomology). This map has divisible kernel; however, $Z\mapsto H^3_\nr(Y\times Z,Y;(\Q/\Z)'(2))$ is a normalised motivic birational invariant, hence $H^3_\nr(Y\times X,Y;(\Q/\Z)'(2))$ is killed by $\Tor(X)$ and therefore finite; so \eqref{eq12} is an isomorphism.

Let $M=\Coker\cl^2_{Y\times X,Y}/\tors$; by \cite[Cor. 3.5]{cycletale}, this is actually $\Coker\tilde \cl^2_{Y\times X,Y}$, although we won't use it. The composition of the map $\gamma$ of \eqref{eq9} with the projection $p:\Coker\cl^2_{Y\times X,Y}\to M$ has image isomorphic to $H^2_\tr(Y)\otimes(\NS(X)/\tors)$. 

I claim that $p\circ \gamma$ is surjective. To see this, choose a retraction $\rho$ of the map $\theta$ in Diagram \eqref{eq11}; composing $\rho\circ \cl^2_{Y\times X,Y}$ with the projection to $\Coker \psi$, we get an induced map 
\[CH^2(X_{k(Y)})/CH^2(X)\otimes \hat{\Z}'\to \Coker \psi\] 
whose composition with $\Coker\psi\to H^2_\tr(Y)\otimes(\NS(X)/\tors)$ is $0$ since $CH^2(X_{k(Y)})/CH^2(X)$ is torsion. This shows that $\rho$ induces a map
\[\bar \rho:\Coker  \cl^2_{Y\times X,Y}\to H^2_\tr(Y)\otimes(\NS(X)/\tors)\]
factoring through a left inverse of the inclusion $H^2_\tr(Y)\otimes(\NS(X)/\tors)\allowbreak\inj M$ induced by $\gamma$. But $\gamma\otimes \Q$ is an isomorphism, since $\Ker\phi$ and $\Coker \phi$ are torsion; therefore $p\circ \gamma$ is surjective as claimed.

Chasing in \eqref{eq13} with this information and using the isomorphism \eqref{eq12} now yields the exact sequence \eqref{eq7}.
\end{proof} 

\begin{cor}\label{c1} a) Suppose $Y=X$. Then we have a commutative diagram of short exact sequences:
\[\begin{CD}
0&\to& CH^2(X\times X)\otimes \hat\Z' @>\cl^2_{X\times X}>> H^4(X\times X)&\to& H^3_\nr(X\times X,(\Q/\Z)'(2))&\to& 0\\
&& @VVV @VVV ||\\
0&\to&  CH^2(X_{k(X)})/CH^2(X)[1/p] @>\phi >> \Tor(H^2(X),H^3(X))^2&\to& H^3_\nr(X\times X,(\Q/\Z)'(2))&\to& 0.
\end{CD}\]
b) In particular, the first map of \eqref{eq0} (for $Y=X$) has $p$-primary torsion kernel, and $\Tor^p(X)=\exp^p(\NS(X)_\tors)$.
\end{cor}

\begin{proof} Indeed, we have $\Pic^0(X)=H^1(X)=H^2_\tr(X)=0$, and Theorem \ref{p2} boils down to the injectivity of $\cl^2_{X\times X,X}$ and $\phi$, plus an isomorphism $H^3_\nr(X\times X,X;(\Q/\Z)'(2)) \iso \Coker \phi$. But $H^3_\nr(X,(\Q/\Z)'(2))=0$, hence $H^3_\nr(X\times X,(\Q/\Z)'(2)) \iso H^3_\nr(X\times X,X;(\Q/\Z)'(2))$.
\end{proof}

\begin{cor}\label{c2} If $Y$ is a curve, we have a short exact sequence
\[0\to CH^2(X_{k(Y)})_\tors[1/p]\to H^1(Y)\otimes H^3(X)\to \Coker \cl^2_{Y\times X} \to 0.\]
\end{cor}

\begin{proof} In this case, $H^2_\tr(Y)=0$ and the target of $\phi$ is $0$.
\end{proof}

\begin{rks}\label{r1} a) The special case $\NS(X)_\tors=0$ and $\car k=0$ of Corollary \ref{c1} b) was proven in \cite[Cor. 1.10]{auel} and \cite[Prop. 2.2]{voisin2}. As Colliot-Thélène points out, the methods of \cite{ctr} imply that for any smooth projective $k$-variety $X$ with $b^1=0$ and $b^2=\rho$, $\Ker(CH^2(X_K)\allowbreak \to CH^2(X_{\bar K}))$ is killed by $\exp(\NS(X)_\tors)\cdot \exp(\Br(X))$ (see Theorem \ref{torsion}).\\
b) In the first version of this paper, I had proven Corollaries \ref{c1} and \ref{c2} but had doubts on the finiteness of $CH_0(X_{k(Y)})_\tors$ in general. Colliot-Thélène provided a proof based on his 1991 CIME course \cite{cime}, see Theorem \ref{finitude}. This encouraged me to find a proof in the spirit of this note, and Theorem \ref{p2} is the result. Note that the group $\Theta$ appearing in \cite[Th. 7.1]{cime}  coincides with $H^4_\et(X,\Z(2))_\tors$. In this spirit, a weaker analogue of \cite[Th. 7.3]{cime} is the following fact: for any field $F$, the functor $\Sm^\proj(F)\ni Z\mapsto \Ker(H^4_\et(Z,\Z(2))\to H^4_\et(Z_{\bar F},\Z(2))$ is a normalised motivic birational invariant (indeed, the map $H^2_\et(Y,\Z(1))\to H^2_\et(Y_{\bar F},\Z(1))$ is injective for any smooth projective $Y$). As a consequence, $\Ker(H^4_\et(X,\Z(2))\to H^4_\et(X_{\bar F},\Z(2))$ is killed by $\Tor(X)$ if $X$ has a trivial birational motive.
\end{rks}

\appendix

\renewcommand{\car}{\operatorname{car}}

\renewcommand{\H}{{\mathcal H}}
\renewcommand{\L}{{\mathbb L}}
\newcommand{\oi}{\hskip1mm {\buildrel \simeq \over \rightarrow} \hskip1mm}
\newcommand{\ovF}{{\overline {\F}}}
\newcommand{\ovK}{{\overline K}}
\newcommand{\ovk}{{\overline k}}
\newcommand{\ovV}{{\overline V}}
\newcommand{\ovX}{{\overline X}}
\newcommand{\K}{{\mathcal K}}
\newcommand{\Cores}{{\operatorname{Cores}}}
\newcommand{\cont}{{\operatorname{cont}}}
\newcommand{\cd}{{\operatorname{cd }}}
\newcommand{\bir}{{\operatorname{bir}}}
\newcommand{\inv}{{\operatorname{inv}}}
\newcommand{\Tate}{{\operatorname{Tate}}}
\newcommand{\Norm}{{\operatorname{Norm}}}

\newcommand{\cqfd}
{%
\mbox{}%
\nolinebreak%
\hfill%
\rule{2mm}{2mm}%
\medbreak%
\par%
}
\newfont{\gothic}{eufb10}


\renewcommand{\qed}{{\hfill$\square$}}

\theoremstyle{theorem}
\newtheorem{theo}[thm]{Th\'{e}or\`{e}me}
\newtheorem{lem}[thm]{Lemme}
\theoremstyle{definition}
\newtheorem{defi}[thm]{D\'efinition}
\theoremstyle{remark}
\newtheorem{rema}[thm]{Remarque}
\newtheorem{remas}[thm]{Remarques}
\newtheorem{propr}[thm]{Propri\'et\'es}

\newcommand{\bthe}{\begin{theo}}
\newcommand{\ble}{\begin{lem}}
\newcommand{\bpr}{\begin{prop}}
\newcommand{\bco}{\begin{cor}}
\newcommand{\bde}{\begin{defi}}
\newcommand{\ethe}{\end{theo}}
\newcommand{\ele}{\end{lem}}
\newcommand{\epr}{\end{prop}}
\newcommand{\eco}{\end{cor}}
\newcommand{\ede}{\end{defi}}

\newcommand{\F}{{\mathbb F}}

\newcommand{\Y}{{\mathcal Y}}
\newcommand{\bP}{{\mathbb P}}

\def\K{{\overline K}}
\def\k{{\overline k}}

\section{Cycles de codimension deux,  compl\'ement \`a deux anciens articles}

\hfill par Jean-Louis Colliot-Thélène

\subsection{Introduction}

On donne des cons\'equences faciles de r\'esultats \'etablis  dans \cite{ctr} (avec W.  Raskind)
et  dans  le rapport de synth\`ese \cite{cime}, en particulier dans une section o\`u je d\'eveloppais
des arguments de S. Saito et de P. Salberger.

\subsection{Notations et rappels}

Pour simplifier les \'enonc\'es, on se li\-mi\-te ici aux
vari\'et\'es d\'efinies sur un corps de caract\'eristique nulle. On note $\overline k$ une cl\^oture
alg\'ebrique de $k$.
Pour une telle $k$-vari\'et\'e $X$, suppos\'ee projective, lisse, g\'eom\'etriquement connexe
sur le corps $k$, on note $\X=X\times_{k}{\overline k}$. On note $b_{i}$
 le $i$-i\`eme nombre de Betti $l$-adique de $\X$. On sait qu'il est ind\'ependant du nombre premier $l$.
 On note $\rho$ le rang du groupe de N\'eron-Severi g\'eom\'etrique $\NS(\X)$.
 Pour tout entier $i$, on note  ici  $H^{i}(\X,\hat{\Z}(j)): = \prod_{l} H^{i}_{\et}(\X,\Z_{l}(j))$.
 Le sous-groupe de torsion $H^{i}(\X,\hat{\Z}(j))_{\tors}$ est fini. On note $e_{i}$ son exposant.
 Pour $k=\C$ le corps des complexes, $H^{i}_{Betti}(X(\C),\Z) \otimes \Z_{l} \simeq H^{i}_{\et}(X,\Z_{l})$.
On sait que l'on a un isomorphisme de groupes finis $\NS(\X)_{\tors} = H^{2}(\X,\hat{\Z}(1))_{\tors}$.
Le groupe de Brauer ${\rm Br}(\X)$  de $\X$ est extension du groupe fini $H^{3}(\X,\hat{\Z}(1))_{\tors}$
par $(\Q/\Z)^{b_{2}-\rho}$.
La condition $H^1(X,O_{X})=0$ \'equivaut \`a $b_{1}=0$.
  La condition $H^2(X,O_{X})=0$ \'equivaut (th\'eorie de Hodge) \`a $\rho= b_{2}$,
  c'est-\`a-dire \`a la finitude du groupe de Brauer de $\X$.
Pour $X$ une vari\'et\'e lisse,   on note $CH^{i}(X)$ le groupe de Chow des cycles de codimension $i$ de $X$.
 Pour $X$ une vari\'et\'e projective, on note $CH_{i}(X)$ le groupe de Chow des cycles de dimension $i$ de $X$. 
  
  \subsection{Exposant de torsion}
  
 L'\'enonc\'e suivant aurait pu \^etre inclus dans \cite{ctr}.
 Comme indiqu\'e formellement ci-dessus, l'entier $e_{i}$
 est l'an\-nu\-la\-teur de la torsion du $i$-\`eme groupe de cohomologie enti\`ere.
 
\begin{theo}\label{torsion}
 Soit $k$ un corps de caract\'eristique z\'ero.
Soit $X$ une $k$-vari\'et\'e projective, lisse, connexe, satisfaisant $X(k) \neq \emptyset$.
Supposons que le r\'eseau  $\NS(\X)/\tors$ admet une base
globalement respect\'ee par le 
groupe de Galois absolu de $k$.  

(a) Supposons $b_{1}=0$ et $\rho=b_{2}$.
Alors  le  groupe de torsion
${\rm Ker} [CH^2(X)\allowbreak \to CH^2(X_{\overline k})]$ est annul\'e 
par le produit  $e_{2}.e_{3}$, qui est aussi le produit de l'exposant de
  $\NS(\X)_{\tors}$ et de l'exposant du groupe $\Br(\X)$.
  
  (b) Si de plus $b_{3}=0$, alors  $CH^2(X)_{\tors}$ est annul\'e par $e_{2}.e_{3}.e_{4}$.
  \end{theo}

\begin{proof}[Démonstration]  Il suffit de suivre les d\'emonstrations du \S 3 de \cite{ctr}.
On note $H^{i}(k,\bullet)$ les groupes de cohomologie galoisienne.

Sous l'hypoth\`ese $H^1(X,O_{X})=0$, le th\'eor\`eme 1.8 de \cite{ctr} donne une suite exacte
de modules galoisiens
$$ 0 \to D_{0} \to H^0(\X,{\mathcal K}_{2}) \to H^{2}(\X,\hat{\Z}(1))_{\tors}  \to 0$$
o\`u  $D_{0}$ est uniquement divisible.
Le groupe $K_{2}\k$ est uniquement divisible. On a la suite exacte
$$ 0 \to H^0(\X,{\mathcal K}_{2})/K_{2}\k \to K_{2}\k(X)/K_{2}\k \to K_{2}\k(X)/H^0(\X,{\mathcal K}_{2}) \to 0.$$
Comme on a suppos\'e $X(k) \neq \emptyset$, on a  $H^1(k, K_{2}\k(X)/K_{2}\k )=0$ \cite[Theorem 1]{ct83}.
On voit alors que le groupe $H^1(k, K_{2}\k(X)/H^0(\X,{\mathcal K}_{2}))$ est un sous-groupe de
$H^2(k, H^{2}(\X,\hat{\Z}(1))_{\tors})$ et donc est annul\'e par $e_{2}$.

Sous les deux hypoth\`eses $H^2(X,O_{X})=0$ et $H^1(X,O_{X})=0$ (cette derni\`ere garantissant $Pic(\X)=\NS(\X)$),
   le th\'eor\`eme 2.12  de \cite{ctr} donne une suite exacte
 de modules galoisiens 
 $$ 0 \to D_{1} \to \NS(\X) \otimes {\overline k}^* \to H^1(\X,{\mathcal K}_{2}) \to [D_{2} \oplus H^{3}(\X,\hat{\Z}(2))_{\tors}] \to 0,$$
 o\`u $D_{1}$ et $D_{2}$ sont uniquement divisibles.
L'hypoth\`ese que l'action du groupe de Galois sur  $\NS(\X)/\tors$ est triviale assure via le th\'eor\`eme 90 de Hilbert
que l'on a $H^1(k, \NS(\X) \otimes {\overline k}^* )=0$. 
De la suite exac\-te ci-dessus on d\'eduit  que $H^1(k,H^1(\X,{\mathcal K}_{2}))$ est un sous-groupe de
$H^1(k, H^{3}(\X,\hat{\Z}(2))_{\tors})$ et donc est annul\'e par $e_{3}$.

La proposition 3.6 de \cite{ctr} fournit  une suite exacte
\begin{multline*}
 H^1(k, K_{2}\k(X)/H^0(\X,{\mathcal K}_{2})) \to {\rm Ker}[CH^2(X) \to CH^2(\X)] \\ 
 \to H^1(k,H^1(\X,{\mathcal K}_{2})).
 \end{multline*}

On voit donc que ${\rm Ker}[CH^2(X) \to CH^2(\X)]$ est annul\'e par le produit $e_{2}.e_{3}$.

 Par Bloch et Merkurjev-Suslin, $CH^2(\X)_{\tors}$ est un sous-quotient de $H^3_{\et}(\X, \Q/\Z(2))$
 \cite[Th\'eor\`eme 3.3.2]{cime}.
Si    $b_{3}=0$, alors $CH^2(\X)_{\tors}$,
 est un sous-quotient de $H^4(\X, \hat{\Z}(2))_{\tors}$,
  d'exposant $e_{4}$.  Sous les hypoth\`eses du th\'eor\`eme, on obtient  alors que
 $CH^2(X)_{\tors}$ est annul\'e par $e_{2}.e_{3}.e_{4}$.
\end{proof}

\begin{remas}

1) Soit $Y$ une vari\'et\'e projective et lisse sur le corps des complexes $\C$ satisfaisant 
les hypoth\`eses du th\'eor\`eme.
Pour tout corps $k$ contenant $\C$, le th\'eor\`eme s'applique \`a la $k$-vari\'et\'e $X=Y\times_{\C}k$.
L'hypoth\`ese sur l'action galoisienne est alors automatiquement satisfaite pour la $k$-vari\'et\'e $X$, car on a $\NS(Y) = \NS(\X)$.
 
 2) Lorque $e_{2}=1=e_{3}$,  l'\'enonc\'e (a) est le th\'eor\`eme 3.10 b) de \cite{ctr}.

3) Si $X$ est une surface,  $e_{4}=1$, et $b_{1}=b_{3}$. En outre, $e_{2}=e_{3}$.
  Sous les hypoth\`eses du th\'eor\`eme,   on trouve que le groupe $CH^2(X)_{\tors}=CH_{0}(X)_{\tors}$
 est annul\'e   par  le carrr\'e de l'exposant de la torsion de $\NS(\X)$. 
 
  \end{remas}
  
  \subsection{Finitude}
  
On  utilise  ici les notations et r\'esultats du \S 7 de  \cite{cime}.

\begin{theo}\label{conoyau}
Soient $k$ un corps de caract\'eristique z\'ero et $\overline k$ une cl\^oture alg\'ebrique.
Soit $X$ une $k$-vari\'et\'e projective et lisse, g\'eom\'etriquement int\`egre.
Notons $\X=X\times_{k}{\overline k}$.
Notons $b_{i} \in \N$  les nombres de Betti $l$-adiques de $\X$
et $\rho=rang(\NS(\X))$.
Supposons $H^1(X,O_{X})$=0, ce qui \'equivaut \`a $b_{1}=0$.
Supposons aussi  $H^2(X,O_{X})=0$, ce qui \'equivaut \`a $\rho=b_{2}$.
Supposons $b_{3}=0$.
Alors le conoyau de l'application
$$H^3_{\et}(k,\Q/\Z(2)) \oplus [H^1(X,{\mathcal K}_{2}) \otimes \Q/\Z] \to H^3_{\et}(X,\Q/\Z(2))$$
est d'exposant fini.
\end{theo}

\begin{proof}[Démonstration]  L'hypoth\`ese $b_{3}=0$ implique que le groupe
$H^3_{\et}(\X,\Q/\Z(2))$ s'identifie au groupe fini
$H^4_{\et}(\X, \hat{\Z}(2))_{\tors}$.
L'\'enonc\'e est alors une cons\'equence
imm\'ediate du Th\'eor\`eme 7.3 de   \cite{cime}, auquel je renvoie pour les notations. \end{proof}

\begin{theo}\label{factorisation}
Soient $k$ un corps de caract\'eristique  z\'ero et $\overline k$ une cl\^oture alg\'ebrique.
Soit $X$ une $k$-vari\'et\'e projective et lisse, g\'eom\'etriquement int\`egre.
Notons $\X=X\times_{k}{\overline k}$. Supposons que chacun des entiers  $b_{1}$, $b_{2}-\rho$ et $b_{3}$
associ\'es \`a $\X$
est nul. Supposons $X(k)\neq \emptyset$. Alors il existe un entier $N>0$ annulant le groupe $CH^2(X)_{\tors}$
et tel que pour tout entier $n>0$ multiple de $N$, l'application  
$$ CH^2(X)_{\tors} \to CH^2(X)/n \to H^4_{\et}(X,\mu_{n}^{\otimes 2})$$
compos\'ee de la projection naturelle et de l'application classe de cycle en cohomologie \'etale
est injective.
\end{theo}

\begin{proof}[Démonstration] Il suffit de combiner le th\'eor\`eme \ref{conoyau}   avec le th\'eor\`eme 7.2
de \cite{cime}. \end{proof}

\begin{rema}
 Si $X$ est une surface, l'hypoth\`ese $b_{3}=0$ est impliqu\'ee par $b_{1}=0$.
\end{rema}

On dit qu'un corps $k$ de caract\'eristique z\'ero est \`a cohomologie galoi\-sienne finie
si pour tout module fini galoisien $M$ sur $k$, tous les groupes de cohomologie galoisienne
  $H^{i}(k,M)$ sont finis. Parmi les corps de caract\'eristique z\'ero satisfaisant cette propri\'et\'e,
  on trouve :  les corps alg\'ebriquement clos, les corps r\'eels clos, les corps $p$-adiques,
  les corps de s\'eries formelles it\'er\'ees sur un des corps pr\'ec\'edents.

\begin{theo}\label{finitude}
Soit $k$ un corps de caract\'eristique z\'ero \`a cohomologie galoisienne finie.
Soit $K$ un corps de type fini sur $k$.
Soit $X$ une $K$-vari\'et\'e projective et lisse satisfaisant $X(K)\neq \emptyset$.
Notons $\X=X\times_{K}\K$. 
Supposons que chacun des entiers  $b_{1}$, $b_{2}-\rho$ et $b_{3}$ associ\'es \`a $\X$ est nul.
Alors le groupe $CH^2(X)_{\tors}$ est fini.
\end{theo}

\begin{proof}[Démonstration] D'apr\`es le th\'eor\`eme \ref{factorisation},  il existe un entier $n>0$ tel que
le groupe $CH^2(X)_{\tors}$ s'identifie 
\`a un sous-groupe de l'image de l'application classe de cycle
$$CH^2(X)/n \to H^4_{\et}(X,\mu_{n}^{\otimes 2}).$$
Soit $Y$ une $k$-vari\'et\'e int\`egre de corps des fonctions $K$. 
Quitte \`a res\-treindre la $k$-vari\'et\'e $Y$ \`a un ouvert non vide convenable,
il existe un $Y$-sch\'ema int\`egre,  projectif et lisse ${\mathcal X} \to Y$ dont la fibre
g\'en\'erique est la $K$-vari\'et\'e $X$.
L'application de restriction
$CH^2({\mathcal X}) \to CH^2(X)$ est surjective, et les
applications classe de cycle
$CH^2(X )/n \to H^4_{\et}(X,\mu_{n}^{\otimes 2})$
et 
$CH^2({\mathcal X})/n \to H^4({\mathcal X},\mu_{n}^{\otimes 2})$
sont compatibles.
L'image de 
$CH^2(X)/n \to H^4_{\et}(X,\mu_{n}^{\otimes 2})$
est donc dans l'image de la restriction
$ H^4({\mathcal X},\mu_{n}^{\otimes 2}) \to  H^4_{\et}(X,\mu_{n}^{\otimes 2}).$
Sous les hypoth\`eses du th\'eor\`eme, les groupes $H^{i}({W},\mu_{n}^{\otimes j})$
sont finis pour toute  vari\'et\'e $W$ de type fini sur $k$, en particulier 
$H^4({\mathcal X},\mu_{n}^{\otimes 2})$ est fini. On conclut que $CH^2(X)_{\tors}$ est fini.
\end{proof}

\begin{rema}
Si $X$ est une $K$-surface, $b_{1}=b_{3}$ et 
 l'hypoth\`ese est simplement
que $b_{1}=0$ et $b_{2}-\rho=0$, et la conclusion est que 
$CH_{0}(X)_{\tors}$ est fini.
 \end{rema}

\end{document}